\documentclass[11pt]{amsart}
\usepackage{amsmath,amsthm,amssymb,graphicx,enumerate,natbib,color,url}
\usepackage{pdfsync}
\newtheorem{theorem}{Theorem}[section]

\theoremstyle{remark}

\numberwithin{equation}{section}
\begin{document}

\author[S. Ghosh]{Souvik Ghosh}
\address{Department of Statistics\\
Columbia University \\
New York, NY 10027.}
\email{ghosh@stat.columbia.edu}

\author[S. I. Resnick]{Sidney I. Resnick}
\address{School of Operations Research and Information Engineering\\
  Cornell University\\ Ithaca, NY 14853.}
\email{sir1@cornell.edu}

\date{\today}

\title[Linearity of ME Plot]{When does the Mean Excess Plot Look Linear?}  

\thanks{
S. I. Resnick was partially supported by ARO Contract W911NF-10-1-0289
at Cornell University. 
}

\begin{abstract} In risk analysis, the mean excess plot is a commonly
  used exploratory plotting technique for confirming iid data is
  consistent with a generalized Pareto assumption for the underlying
  distribution since in the presence of such a distribution,
  thresholded data have a mean excess plot that is roughly
  linear. Does any other class of distributions share this linearity
  of the plot? Under some extra assumptions, we are able to conclude
  that only the generalized Pareto family has this property.
\end{abstract}

\maketitle

\begin{section}{Introduction} \label{sec:intro}
The mean excess (ME) plot is a  diagnostic tool commonly used in risk
analysis to justify fitting a generalized Pareto distribution (GPD)
\begin{equation}\label{eq:GPD}
	G_{ \xi ,\beta }(x)= \left\{   \begin{array}{ll}  1-(1+  \xi x/ \beta  )^{ -1/\xi }  &\mbox{ if }\xi \neq 0 \\ 1-\exp(-x/\beta ) & \mbox{ if } \xi =0 \end{array}  \right.
\end{equation}
to
excesses over a large threshold. 
In \eqref{eq:GPD}
 $ \beta >0$, and $ x\ge 0$ when $ \xi \ge 0$ and $ 0\le x\le
 -\beta/\xi $ if $ \xi <0$. The parameters $ \xi $ and $ \beta $ are
the \emph{shape} and \emph{scale} parameters
 respectively. For a Pareto distribution, the tail index $ \alpha $ is
 just the reciprocal of $ \xi$ when $ \xi>0$. A special case is when $
 \xi =0$ and in this case the GPD is the same as the exponential
 distribution with mean $ \beta $.  The use of the diagnostic is
 described in \citet{embrechtskluppelbergmikosch:1997,
embrechtsmcneilfrey:2005, davison1990meo,
ghosh:resnick:2010}.

For a random
variable $X$ satisfying $EX^+<\infty$ with distribution function
$F(x)$ with right endpoint $x_F$  and tail $\bar
F(x)=1-F(x)$, the ME function  is
\begin{equation}\label{eq:meanexcess}
	M(u):= E\big[  X-u|X>u \big]=\frac{\int_u^{x_F} \bar F(s)
          ds}{\bar F(u)}, \qquad u<x_F.
\end{equation}
The ME function is also known as the {\it mean residual life function\/},
especially in
 survival analysis 
\citep{Benktander:1960}. See  \cite{hall1981mrl} for a
discussion of properties.
 Table 3.4.7 in \cite[p.161]{embrechtskluppelbergmikosch:1997} gives 
the mean excess function for standard distributions. The important
fact is that for a GPD distribution with $\xi <1$, the ME function is linear with
positive, negative or zero slope according to whether $0<\xi<1, \xi<0$
or $\xi=0.$ More precisely,
if the 
random variable $ X$ has GPD distribution $G_{ \xi ,\beta }$, we have $E(X)<\infty$ iff
$ \xi <1$ and in this case, the ME function of $X$ is 
\begin{equation}\label{eq:megdp}
	M(u)= \frac{  \beta}{1-\xi  }+ \frac{ \xi}{1-\xi}u,
\end{equation}
where $ 0\le u< \infty $ if $ 0\le \xi<1$ and $ 0\le u\le -\beta /\xi$
if $ \xi<0$. 
 In fact, the linearity of the mean excess function
characterizes the GPD class.
See 
\cite{embrechtsmcneilfrey:2005, embrechtskluppelbergmikosch:1997,
davison1990meo}.
This leads to the diagnostic of exploring the validity of the GPD
assumption (or more broadly whether the underlying distribution is in 
the domain of attraction of a GPD distribution) by plotting an empirical
estimate of the ME function called the ME plot and observing if  (a) 
the
plot looks linear,
at least {after some threshold},
 and  if so, (b) whether the slope is positive, negative or zero.

Given an independent and identically distributed (iid) sample $X_1,\dots,X_n$ from
$F(x)$, a natural estimate of  $M(u)$ 
is the empirical ME function
$ \hat M(u)$ defined as
\begin{equation}\label{eq:empiricalmeanexcess}
	\hat M(u)= \frac{ \sum_{ i=1}^{ n}(X_{ i}-u)I_{ [ X_{ i}>u]}}{ \sum_{ i=1}^{ n}I_{ [ X_{ i}>u]}}, \ \ \ \ u\ge 0.
\end{equation}
The  {\it ME plot\/} is the plot
of the points $ \{   (X_{ (k)},\hat M(X_{ (k)})):1< k\le n  \}$, where
$ X_{ (1)}\ge X_{ (2)}\ge \cdots\ge X_{ (n)}$ are the order statistics
of the data. If the ME plot  is close to linear for
high values of the threshold  then there is no evidence against use
of a GPD
model for the thresholded data.  \citet{ghosh:resnick:2010} offered an
explanation of why the ME plot from a GPD distribution with $\xi <1$
should appear to be linear by considering the ME plot from a sample of
size $n$ as a  random closed set in $\mathbb{R}^2$ indexed by $n$ and
showing convergence as $n \to \infty$ to a line segment in the Fell topology on the
space of closed subsets of $\mathbb{R}^2$. For information about the
Fell topology, Hausdorf metric and the topological space of closed
subsets see \citet{
matheron1975rsa,
beer:1993, molchanov:2005,
ghosh:resnick:2010,
das2008qpr}. {Of course, there are considerable 
practical difficulties interpreting the 
phrase {\it close to linear\/}. \cite{DasGhosh:1000} attempt to
overcome this difficulty by using weak limits of these plots (when $
0<\xi<1$)  to construct confidence bands around the
observed plot.} 

The results in \citet{ghosh:resnick:2010} say that if the underlying
distribution of the underlying sample is in a domain of attraction,
then the ME plot of the random sample should be linear. We state this
precisely below. So these results state that approximate linearity of
the ME plot is consistent with GPD or domain of attraction
assumptions. However, these results do not rule out some other
disjoint class
of distributions giving a ME plot which is approximately linear. Thus
it is the converse of the implications in \citet{ghosh:resnick:2010}
which are the subject of this paper: If the ME plot is approximately
linear, does this imply the underlying distribution is in a domain of
attraction?  We can give an affirmative answer subject to some
assumptions. These converse investigations are related to some skilled
investigations of {David Mason}; see for example \citet{mason:1982}. 

\subsection{Background}
For background on GPD distributions and domains of attraction see
\citet{
dehaan:1970,
resnick2006htp,
resnick:1987,
dehaan:ferreira:2006,
embrechtskluppelbergmikosch:1997}.  References for random closed sets
have already been given. The
class of regularly varying distributions with index $\xi \in
\mathbb{R}$ is denoted by $RV_\xi$. To understand what converses are
required, we restate the main sufficiency results from
\citet{ghosh:resnick:2010}. 
For these results, $\mathcal{F}$ is the
space of closed subsets of $\mathbb{R}^2$ with the Fell topology and
$\stackrel{P}{\to}$ means convergence in probability in $\mathcal{F}$. Let $X_1,\dots,X_n$ be iid with common distribution
$F$,
order statistics $ X_{ (1)}\ge X_{ (2)}\ge \cdots\ge X_{ (n)}$, 
and $k=k_{ n}$ is {\it any\/} sequence satisfying $k\to \infty$ but $k/n \to 0,$
as $n\to \infty$.  
\begin{itemize}
\item  If $ F$ satisfies $ \bar F\in RV_{ -1/\xi }$ with $ 0<\xi<1$,  then in $ \mathcal{F}$,
 \begin{equation}\label{eq:lim:pos} 
 	\mathcal{S}_{ n}:=\frac{ 1}{X_{ (k)}  }  \left\{   \big(X_{  (i)},\hat M(X_{ (i)})\big) :i=2,\ldots, k\right\}   \stackrel{ P}{\longrightarrow } \mathcal{S}:=\Big\{   \Big(t, \frac{ \xi}{1-\xi } t \Big):t\ge 1  \Big\} .
\end{equation}

\item If $ F$ has  finite right end point $ x_{ F}$ and
  satisfies $ 1-F(x_{ F}-x^{ -1})\in RV_{ 1/\xi}$ as $ x\to \infty$
  for some $ \xi<0$, then in $ \mathcal{F}$,
 \begin{align} 
 	\mathcal{S}_{ n}:&= \frac{ 1}{X_{ (1) } -X_{ (k)}}\left\{
          \Big(X_{ (i)}-X_{ (k)} , \hat M(X_{ (i)})  \Big):1<i\le k
        \right\} \nonumber \\
	&\stackrel{ P}{\rightarrow  }   \mathcal{S}:=\Big\{  \Big(
        t, \frac{ \xi}{1-\xi }(t-1)  \Big):0\le t\le 1    \Big\} . \label{eqn:sidxi<0}
\end{align}

\item If $F$ has right end point $x_F \leq \infty$ and is in the
maximal  domain of attraction of the Gumbel distribution, then
in $ \mathcal{F}$,
\begin{equation}\label{eqn:sidxi=0}
 	\mathcal{S}_{ n}:=\frac{ 1}{X_{ (\lceil k/2 \rceil)}-X_{ (k)} }\left\{\Big( X_{ (i)}-X_{ (k)}  , \hat M(X_{ (i)})  \Big):1<i\le k\right\} \stackrel{ P}{\longrightarrow  } \mathcal{S}:=\Big\{  \big( t, 1 \big):t\ge 0    \Big\}.
\end{equation}
\end{itemize}

\subsection{Miscellany}
Throughout this paper we will take $ k:=k_{ n}$ to be a sequence
increasing to infinity such that $ k_{ n}/n\to 0$.  For a distribution
function $F(x)$ we write $\bar F(x)=1-F(x) $ for the tail and the
quantile function is 
\[
	b(u)= F^\leftarrow (1-\frac 1u)=\inf\{s:F(s)\geq 1-\frac 1u\}=\Bigl(\frac{1}{1-F}\Bigr)^\leftarrow (u)
\]
where  $ F^{ \leftarrow }(u):= \inf\{   x:F(x)\ge u  \}$ is the left-continuous inverse of $ F$. 

  A function $U:(0,\infty)\mapsto \mathbb{R}_+$ is regularly varying with index $\rho \in \mathbb{R}$, written $U\in RV_\rho$, if
\[
	\lim_{t\to\infty} \frac{U(tx)}{U(t)}=x^\rho,\quad x>0.
\]

A nondecreasing function $ U$ defined on an interval $ (x_{ l},x_{ 0})$ is $ \Gamma $-varying, written $ U\in \Gamma $, if $ \lim_{ x\rightarrow x_{ 0}} U(x)=\infty$ and there exists a positive function $ f$ defined on $ (x_{ l},x_{ 0})$ such that for all $ x$
\[ 
 	\lim_{ t \rightarrow x_{ 0}  } \frac{ U(t+xf(t))}{ U(t) }=e^{ x}.   
\]
The function $ f$ is called an auxiliary function. 

A nonnegative, nondecreasing function $ V$ defined on $ (x_{ l},\infty)$ is $ \Pi$-varying, written $ V\in \Pi$, if there exists $ a(t)>0, b(t)\in \mathbb{R}$ such that for $ x>0$
\[ 
 	\lim_{ t \rightarrow \infty  } \frac{ V(tx)-b(t)}{ a(t) } =\log x.   
\]
The function $ a(t)$ is unique up to asymptotic equivalence and is
called an auxiliary function. See \citet{dehaan:1970,
bingham:goldie:teugels:1989,
dehaan:ferreira:2006,
resnick:1987} 
for details on regular variation, $
\Gamma$-variation and $ \Pi$-variation. 
\end{section}

\begin{section}{What if the ME plot converges?}\label{sec:main}
We now attempt to draw conclusions from the assumption that the ME
plot converges as $n\to\infty$. We need to phrase what we mean by {\it
  convergence of the
ME plot\/} slightly differently in the three cases. For each case,
there is an issue to resolve about convergence of random sets in
\eqref{eq:lim:pos}, \eqref{eqn:sidxi<0}, \eqref{eqn:sidxi=0} implying
that a sequence of random variables converges. For instance, how do we
conclude from \eqref{eq:lim:pos} that 
\begin{equation}\label{eqn:sidshow}
\frac{\hat M (X_{(k)}) }{X_{(k)}} \stackrel{P}{\to}
\frac{\xi}{1-\xi}?
\end{equation}

Suppose for $k=k_n \to\infty$ we know that in $\mathcal{F}$ 
$$\mathcal{S}_n:=\{(x_i(n),y_i( n)); 1 \leq i \leq k\} \to
\mathcal{S}:=\{(x,cx): x\geq 1\}$$
for $c>0$, and $x_i(n) \geq 1$, for all $1 \leq i \leq k$ and $n\geq 1$. Then using, for example, \citet[Lemma
2.1.2]{das2008qpr}, we have for large $M>0$,
$$\mathcal{S}_n ^M:=\mathcal{S}_n \cap [0,M]^2
\to \mathcal{S}^M :=\mathcal{S} \cap [0,M]^2,$$
and convergence in the Fell topology reduces to convergence with
respect to the Hausdorf metric in the compact space $[0,M]^2$.
Since $(1,c) \in \mathcal{S}^M,$ there exist $(x_{i'}(n),y_{i'}(n)) \to
(1,c)$ in  $\mathcal{R}^2$ \citep{matheron1975rsa,
das2008qpr, ghosh:resnick:2010}. Thus,
$$\wedge_{i=1}^k x_i(n) \leq x_{i'}(n) \to 1.
$$
Enclose $\mathcal{S}^M $ in a $\delta$-neighborhood $(\mathcal{S}^M
)^\delta$ and for sufficiently large $n$, $\mathcal{S}_n^M \subset
(\mathcal{S}^M)^\delta $. Let $x^*(n)=\wedge_{i=1}^k x_i(n)$ be the
$x$-value achieving the minimum and let $y^*(n)$ be the concomitant;
ie, the $y$-value corresponding to $x^*(n)$. Then for large $n$,
$(x^*(n),y^*(n)) \in (\mathcal{S}^M)^\delta$. Since $x^*(n)$ must be
close to $1$, $y^*(n)$ must be close to $c$. This shows
\eqref{eq:lim:pos}  implies \eqref{eqn:sidshow}.

\subsection{Frech\'et case}\label{subsec:frechet}
 \begin{theorem}\label{thm:xipos} 
 Suppose $ X_{ 1},\ldots,X_{ n}$ is an iid sample from a  distribution $ F$ satisfying
 \begin{equation}\label{eq:moment} 
 	E\big[ X_{ 1}^{ 1+\epsilon } \big]    < \infty \ \ \ \ \mbox{ for some } \epsilon >0.
\end{equation}
If for every sequence $ k:=k_{ n} \to \infty$ such that $ n/k\to
\infty$ we have 
\eqref{eq:lim:pos} so that
\begin{equation}\label{eq:MEplotconv} 
 	\frac{ 1}{kX_{ (k+1)} }\sum_{ i=1}^{ k} \Big( X_{ (i)}-X_{
          (k+1)} \Big)    \stackrel{ P}{\to }  \gamma :
        =\frac{\xi}{1-\xi} >0, 
\end{equation}
then $ \bar F\in RV_{- 1-1/\gamma }$, i.e., $ F$ is the maximal domain
of attraction of the Frech\'et distribution. 
\end{theorem}

\begin{proof} 
We first claim that \eqref{eq:MEplotconv} implies $ F$ does not have a finite right end point. Suppose that is not true and there exists $ c\in \mathbb{R}$ such that $ F(c)=1$. Then we must have $ X_{ (k+1)}\stackrel{ P}{\to }  c$.  That will imply $ \hat M(X_{ (k+1)})\to 0$ which contradicts \eqref{eq:MEplotconv}.  Hence $ F$ can not have a finite right end point and in particular we get 
\begin{equation}\label{eq:xkgtr1}
	P[X_{ (k+1)}\le 1]\to 0 \ \ \ \ \mbox{ as }n\to \infty.
\end{equation} 
Next observe  that 
\[ 
 	   \frac{ 1}{kX_{ (k+1)} }\sum_{ i=1}^{ k} \Big( X_{ (i)}-X_{ (k+1)} \Big) = \frac{ 1}{k }\sum_{ i=1}^{ k}\frac{X_{ (i)}}{X_{ (k+1)}}-1
\]
and therefore, using \eqref{eq:MEplotconv} and \eqref{eq:xkgtr1} it follows 
\[ 
 	   V_{ n}:= \frac{ 1}{k }\sum_{ i=1}^{ k}\frac{X_{ (i)}}{X_{ (k+1)}}I_{ [X_{ (k+1)}>1]}\stackrel{ P}{\to }\gamma +1. 
\]
Since $ V_{ n}$ is a nonnegative random random variable its Laplace transforms must also converge: For $ \lambda >0$
\begin{equation}\label{eq:laplace} 
 	E\Big[ \exp \left( -\lambda V_{ n}  \right)  \Big]    \to e^{ -\lambda (\gamma +1)}
\end{equation}
We will obtain a simplified expression for $ E[ \exp \left( -\lambda V_{ n}  \right)  ]$ in the next few steps. We begin by observing that 
\[ 
 	V_{ n}\stackrel{ d}{= }V_{ n}^{ *}:=     \frac{ 1}{k }\sum_{ i=1}^{ k}\frac{X^{ *}_{ (i)}}{X^{ *}_{ (k+1)}}I_{ [X^{ *}_{ (k+1)}>1]} 
\]
where 
\[ 
 	X_{ i}^{ *}=F^{ \leftarrow }(U_{ i}) , \ \ \ \ U_{ 1},\ldots,U_{ n}\sim \mbox{ iid }U[0,1],   
\]
 $ \stackrel{ d}{= } $ denotes equality in distribution. 
Using the fact that conditioned on $ U_{ (k+1)}$ the order statistics $ U_{ (1)},\ldots,U_{ (k)}$ are distributed like iid sample from $ U[U_{ (k+1)},1]$ \citep{maller1984limiting}, we get
\begin{align}\label{eq:conditionally:iid:step1}
	   E\Big[ \exp \left( -\lambda V_{ n}  \right)  \Big] & = E\Big[ \exp \left( -\lambda V_{ n}^{ *}  \right)  \Big]= E\Big[E \Big[  \exp \left( -\lambda V_{ n}^{ *}  \right) \Big| U_{ (k+1)}\Big] \Big] \nonumber\\
	 &= E \left[ \left(\int_{ U_{ (k+1)}}^{ 1} \exp \left( -\lambda \frac{ F^{ \leftarrow }(x)I_{ [X^{ *}_{ (k+1)}>1]}}{kX^{ *}_{ (k+1)} } \right)\frac{ dx}{1-U_{ (k+1)} }\right)^{ k}  \right]\nonumber \\
	 & = E \left[ \left(1- \frac{ 1}{k }\int_{ U_{ (k+1)}}^{ 1}k\left(1- \exp \left( -\lambda \frac{ F^{ \leftarrow }(x)I_{ [X^{ *}_{ (k+1)}>1]}}{kX^{ *}_{ (k+1)} } \right)\right)\frac{ dx}{1-U_{ (k+1)} }\right)^{ k}  \right].
 \end{align}
 Observe that 
 \[ 
 	   \int_{ U_{ (k+1)}}^{ \infty}k\left(1- \exp \left( -\lambda \frac{ F^{ \leftarrow }(x)I_{ [X^{ *}_{ (k+1)}>1]}}{kX^{ *}_{ (k+1)} } \right)\right)\frac{ dx}{1-U_{ (k+1)} }\le k \ \ \ \ \mbox{ almost surely}.
\]
From \cite{hall1979rate} we know that 
\begin{equation}\label{eq:hall:wellner}
	\sup_{ y\ge 0}\left|  \left( 1- \frac{ y}{n } \right)^{ n}I_{ [0,n]}(y) -e^{ -y}  \right|\le \left( 2+\frac{ 1}{n } \right) e^{ -2} \frac{ 1}{n } =o(1)
\end{equation}
and applying this to \eqref{eq:conditionally:iid:step1} we get
\begin{align}\label{eq:step2}
	& E\Big[ \exp \left( -\lambda V_{ n}  \right)  \Big] \nonumber\\
	& = E \left[\exp \left(- \int_{ U_{ (k+1)}}^{ 1}k\left(1- \exp \left( -\lambda \frac{ F^{ \leftarrow }(x)I_{ [X^{ *}_{ (k+1)}>1]}}{kX^{ *}_{ (k+1)} } \right)\right)\frac{ dx}{1-U_{ (k+1)} }\right)  \right]+o(1)\nonumber \\
	& = E \left[\exp \left(- \int_{ U_{ (k+1)}}^{ 1}k\left(1- \exp \left( -\lambda \frac{ F^{ \leftarrow }(x)}{kX^{ *}_{ (k+1)} } \right)\right)I_{ [X^{ *}_{ (k+1)}>1]}\frac{ dx}{1-U_{ (k+1)} }\right)  \right]+o(1).
\end{align}

Choose $ 0<\epsilon <1$ satisfying \eqref{eq:moment}. We claim that if the sequence $ k$ satisfies $ k\to \infty$ and $ n/k\to \infty$ along with $ n/k^{ 1+\epsilon }\to 0$ (for example  $ k= n^{ 1/(2(1+\epsilon ))}$) then  
\begin{equation}\label{eq:onlyk} 
 	E\Big[ \exp \left( -\lambda V_{ n}  \right)  \Big] =   E \left[\exp \left(- \int_{ U_{ (k+1)}}^{ 1}\lambda \frac{ F^{ \leftarrow }(x)}{X^{ *}_{ (k+1)} } I_{ [X^{ *}_{ (k+1)}>1]}\frac{ dx}{1-U_{ (k+1)} }\right)  \right]+o(1).
\end{equation}
Using \eqref{eq:step2} and the fact that  $  | e^{ -a}-e^{ -b} |\le  | a-b |$ for all $ a,b\ge 0$ it suffices to show that 
\begin{align} \label{eq:step3}
 	G_{ n}:=  &  E\left|  \left[ \int_{ U_{ (k+1)}}^{ 1}k\left(1- \exp \left( -\lambda \frac{ F^{ \leftarrow }(x)}{kX^{ *}_{ (k+1)} } \right)\right)I_{ [X^{ *}_{ (k+1)}>1]}\frac{ dx}{1-U_{ (k+1)} }  \right. \right. \nonumber \\
	  &- \left.\left.  \int_{ U_{ (k+1)}}^{ 1} \lambda \frac{ F^{ \leftarrow }(x)}{X^{ *}_{ (k+1)} } I_{ [X^{ *}_{ (k+1)}>1]}\frac{ dx}{1-U_{ (k+1)} }  \right] \right|\to 0.
\end{align}
Get $ 0<\epsilon \le1$ such that \eqref{eq:moment} holds. Since $  | 1-e^{ -x}-x |\le x^{ 1+\epsilon }$ for all $ x>0$, we obtain a bound for $ G_{ n}$:
\begin{align}\label{eq:Gkbound}
	G_{ n}&\le k \lambda ^{ 1+\epsilon } E \left[   \int_{ U_{ (k+1)}}^{ 1} \left(  \frac{ F^{ \leftarrow }(x)}{k X^{ *}_{ (k+1)} } \right)^{ 1+\epsilon } I_{ [X^{ *}_{ (k+1)}>1]} \frac{ dx}{1-U_{ (k+1)} } \right] \nonumber   \\
	&  \le k^{ -\epsilon }\lambda ^{ 1+\epsilon } E \left[ \int_{ U_{ (k+1)}}^{ 1} F^{ \leftarrow }(x)^{ 1+\epsilon }I_{ [X^{ *}_{ (k+1)}>1]} \frac{ dx}{1-U_{ (k+1)} } \right]  \nonumber\\
	& \le k^{ -\epsilon }\lambda ^{ 1+\epsilon } E\Big[ X_{ 1}^{ 1+\epsilon }I_{ [X_{ 1}>1]} \Big]  E\Big[ \frac{ 1}{1-U_{ (k+1)} } \Big] 
 \end{align}
The form of $ E[(1-U_{ (k+1)})^{ -1}] $ can be easily obtained using the R\'enyi representation \citep[p.110]{resnick2006htp}. Recall that if $ U\sim U[0,1]$ then $ (1-U)^{ -1}\sim$  Pareto(1) and therefore
\[ 
 	\frac{ 1}{1-U_{ (k+1)} } \stackrel{ d}{= } e^{ E_{ (k+1)}}   
\]
where $ E_{ (k+1)}$ is the $ (k+1)$-th order statistic of an iid sample from an exponential distribution with mean 1. Using the R\'enyi representation 
\[ 
 	E\left[\frac{ 1}{1-U_{ (k+1)} }\right]=E\left[ e^{ E_{ (k+1)}} \right]    = E\left[ \prod_{ i=1}^{ n-k-1} e^{ E_{ i}/(n-i+1)}\right]
\]
where $ E_{ 1},\ldots,E_{ n}\sim $iid Exp(1). This implies
\[ 
 	 E\left[\frac{ 1}{1-U_{ (k+1)} }\right]= \prod_{ i=1}^{ n-k-1}\frac{ n-i+1}{n-i }= \frac{ n}{k+1 }  
\]
therefore from \eqref{eq:step3} and \eqref{eq:Gkbound} we get 
\[ 
 	G_{ n}\le    \lambda ^{ 1+\epsilon } E\Big[ X_{ 1}^{ 1+\epsilon }I_{ [X_{ 1}>1]} \Big] \frac{ n}{k^{ 1+\epsilon } } .
\]
Thus, $ G_{ n}\to 0$ if $ n/k^{ 1+\epsilon }\to 0$. This proves the claim \eqref{eq:onlyk}.

Using \eqref{eq:onlyk} and \eqref{eq:laplace} we get that 
\begin{equation}\label{eq:stoch:conv:H} 
 	H(U_{ (k+1)})\stackrel{ P}{\to  }\gamma +1    
\end{equation}
whenever $ n,k\to \infty$ with $ n/k\to 0$ and $ n/k^{ 1+\epsilon }\to \infty$, where
\[ 
 	H(y):=\int_{ y}^{ 1} \frac{ F^{ \leftarrow }(x)}{ F^{ \leftarrow }(y) (1-y)}dx .  
\]
We claim that \eqref{eq:stoch:conv:H} implies 
\begin{equation}\label{eq:conv:H}
	H(y)\to \gamma +1 \ \ \ \ \mbox{ as }y\to 1,
\end{equation}
and we will prove it by contradiction. Write 
\[ 
 	N_{ n}=\frac{ n}{\sqrt{k} } \left( U_{ (k+1)} - \left( 1-\frac{ k}{n } \right) \right)    
\]
and note that $ N_{ n}\Rightarrow N(0,\sigma ^{ 2})$ for some $ \sigma ^{ 2}>0$, see \cite{balkema1975limit}. We know that 
\[ 
 	H \left( \frac{ \sqrt{k}}{n }N_{ n}+1-\frac{ k}{n } \right)    \to \gamma +1 \ \ \ \ \mbox{ as } n,k,\frac{ n}{k },\frac{ k^{ 1+\epsilon }}{n }\to \infty.
\]
If possible suppose \eqref{eq:conv:H} is not true and there exists $ \delta >0$ and $ (   z_{ m}^{ (2)}  )$ such that $ z_{ m}^{ (2)}\to 1$ and 
\[ 
 	\left| H(z_{ m}^{ (2)})-\gamma  \right|   >2\delta .
\]
Since $ H$ is left continuous there exists $ (z_{ m}^{ (1)})$ such that $ z_{ m}^{ (1)}<z_{ m}^{ (2)}$, $ z_{ m}^{ (1)}\to 1$ and 
\[ 
 	\left| H(y)-\gamma  \right|   >\delta \ \ \ \ \mbox{ for all } y\in (z_{ m}^{ (1)},z_{ m}^{ (2)}).
\]
For every $ m\ge 1$ choose $ n(m)$ such that 
\[ 
 	n(m)^{ \epsilon }(1-z_{ m}^{ (1)})^{ 1+\epsilon }   \ge n(m)^{ \epsilon /2} \ \ \ \ \mbox{ and } \ \ \ \ n(m)(z_{ m}^{ (2)}-z_{ m}^{ (1)}) \ge 1+\sqrt{\lfloor n(m)(1-z_{ m}^{ (1)})\rfloor }
\]
and define $ k(n(m))=\lfloor n(m)(1-z_{ m}^{ (1)}) \rfloor$. Then 
\begin{align*}
	& k(n(m))    \ge n(m)^{ \epsilon }(1-z_{ m}^{ (1)})^{ 1+\epsilon }\ge n(m)^{ \epsilon /2}\to \infty,\\
	&  k(n(m))/n(m)\sim 1-z_{ m}^{ (1)}\to 0, \ \ \ \ \mbox{ and}\\
	& k(n(m))^{ 1+\epsilon }/n(m)\sim n(m)^{ \epsilon } (1-z_{ m}^{ (1)})^{ 1+\epsilon } \to \infty.
\end{align*}
Furthermore, we also get 
\[    y_{ m}^{ (1)} := 1- \frac{ k(n(m))}{n(m) } \ge z_{ m}^{ (1)} \ \ \mbox{ and} \ \ y_{ m}^{ (2)}:=   1-\frac{ k(n(m))}{ n(m)} + \frac{ \sqrt{k(n(m))}}{ n(m) } \le z_{ m}^{ (2)}.\]
Now observe that with this construction
\begin{align*}
	& \liminf_{ m \rightarrow\infty  } P \left[ \left| H \left( \frac{ \sqrt{k(n(m))}}{n(m) }N_{ n(m)}+1-\frac{ k(n(m))}{n(m) } \right) -\gamma \right| >\delta  \right]   \\ 
	&  \ge  \liminf_{ m \rightarrow\infty  } P \left[  \frac{ \sqrt{k(n(m))}}{n(m) }N_{ n(m)}+1-\frac{ k(n(m))}{n(m) }\in \left( z_{ m}^{ (1)},z_{ m}^{ (2)} \right)   \right] \\
	& \ge  \liminf_{ m \rightarrow\infty  } P \left[  \frac{ \sqrt{k(n(m))}}{n(m) }N_{ n(m)}+1-\frac{ k(n(m))}{n(m) }\in \left(y_{ m}^{ (1)},y_{ m}^{ (2)} \right)   \right] \\
	& = \liminf_{ m \rightarrow \infty  } P \left[ N_{ n(m)}\in (0,1) \right] = P[N(0,\sigma ^{ 2})\in (0,1)] >0
 \end{align*}
 which contradicts \eqref{eq:stoch:conv:H}. 
 
   Now finally we show that \eqref{eq:conv:H} implies that $ \bar F\in RV_{ -1-1/\gamma }$. It suffices to show that $ b(u):= F^{ \leftarrow }(1-1/u)\in RV_{ \gamma /(\gamma +1)}$. Note that from \eqref{eq:conv:H} we get that 
   \[ 
 	   \frac{ 1}{ y b(y)/y^{ 2}} \int_{ y}^{ \infty} \frac{ b(u)}{u^{ 2} } du \to \gamma +1 \ \ \ \ \mbox{ as }y\to \infty.
\]
   By Karamata's Theorem \citep[Theorem 2.1, p.25]{resnick2006htp} this imples that $ b(u)/u^{ 2}\in RV_{ -(\gamma +2)/(\gamma +1)}$ and hence $ b(u)\in RV_{ \gamma /(\gamma +1)}$. Hence the proof is complete. 
\end{proof}

\subsection{Weibull case}\label{subsec:weibull}
 To deal with this case, we found it necessary to assume a bit more
 than \eqref{eqn:sidxi<0} because we want to replace $X_{(1)}$ by the
 right endpoint of the underlying distribution.

\begin{theorem}\label{thm:xineg} 
Suppose $ X_{ 1},\ldots,X_{ n}$ is an iid sample from a distribution $ F$. 
If there exists  $ \kappa \in \mathbb{R}$ such for every sequence $
k:=k_{ n}\to \infty$ satisfying $ n/k\to \infty$ 
\begin{equation}\label{eq:MEplotconv:xineg} 
 	\frac{ 1}{k(\kappa -X_{ (k+1)}) }\sum_{ i=1}^{ k} \Big( X_{
          (i)}-X_{ (k+1)} \Big)    \stackrel{ P}{\to }  \gamma >0 ,
\end{equation}
then $ \kappa $ is the right end point of $ F$ and $ \bar F(\kappa -1/\cdot)\in RV_{1-1/\gamma }$, i.e., $ F$ is in the maximal domain of the Weibull distribution.
\end{theorem}

The parameter $\gamma$ plays the role of $-\xi/(1-\xi)$ in 
\eqref{eqn:sidxi<0}.

\begin{proof} 
Suppose $ \kappa _{ 0}\in \mathbb{R}\cup \{   \infty  \}$ is the right end point of $ F$. If $ \kappa <\kappa _{ 0}$ then 
\[
	\liminf_{ n \rightarrow\infty  } \kappa -X_{ (k)}<0 \ \ \ \ \mbox{a.s.}
\]
and hence \eqref{eq:MEplotconv:xineg} can not hold. Therefore we must have $ \kappa \ge \kappa _{ 0}$. On the other hand if $ \kappa >\kappa _{ 0}$ then $ \kappa_{ 0} $ is the finite right end point and hence we will have 
\[ 
 	\lim_{ n \rightarrow\infty  } \kappa -X_{ (k)}>0 \ \ \ \ \mbox{ and }\ \ \ \ X_{ (1)}-X_{ (k)} \to 0 \ \ \ \ \mbox{ a.s.  }   
\]
In this case also  \eqref{eq:MEplotconv:xineg} can not hold for $ \gamma >0$. Therefore $ \kappa =\kappa_{ 0}$ must be the finite right end point of the distribution $ F$. Also note that \eqref{eq:MEplotconv:xineg} implies $ 0<\gamma <1$ since $ \kappa -X_{ (i)}\le \kappa -X_{ (k)}$ for all $ 1\le i\le k$. 

The rest of the proof is similar to that of Theorem \ref{thm:xipos}.  Observe that \eqref{eq:MEplotconv:xineg} implies 
\begin{equation}\label{eq:Vn:xineg}
	V_{ n}:= \frac{ Z_{ (k)}}{k }\sum_{ i=1}^{ k} \frac{ 1}{Z_{ (i)} } \stackrel{ P}{\rightarrow  }  1-\gamma,
\end{equation}
where $ Z_{ i}=(\kappa -X_{ i})^{ -1}$. Using the arguments leading to \eqref{eq:step2} we get
\begin{align}\label{eq:step2:xineg}
	& E\Big[ \exp \left( -\lambda V_{ n}  \right)  \Big] \nonumber\\
	& = E \left[\exp \left(- \int_{ U_{ (k+1)}}^{ 1}k\left(1- \exp \left( -\lambda \frac{Z^{ *}_{ (k+1)} }{kF_{ Z}^{ \leftarrow }(x) } \right)\right)I_{ [Z^{ *}_{ (k+1)}>1]}\frac{ dx}{1-U_{ (k+1)} }\right)  \right]+o(1),
\end{align}
where 
\[ 
 	Z_{ i}^{ *}=F_{ Z}^{ \leftarrow }(U_{ i}) , \ \ \ \ U_{ 1},\ldots,U_{ n}\sim \mbox{ iid }U[0,1],   
\]
and $ F_{ Z}$ is the cumulative distribution function of $ Z_{ 1}$. Furthermore, note that
\begin{align} \label{eq:step3:xineg}
 	 &  E\left|  \left[ \int_{ U_{ (k+1)}}^{ 1}k\left(1- \exp \left( -\lambda \frac{Z^{ *}_{ (k+1)} } {k F_{ Z}^{ \leftarrow }(x)}\right)\right)I_{ [Z^{ *}_{ (k+1)}>1]}\frac{ dx}{1-U_{ (k+1)} }  \right. \right. \nonumber \\
	  &- \left.\left.  \int_{ U_{ (k+1)}}^{ 1} \lambda \frac{Z^{ *}_{ (k+1)} }{k F_{ Z}^{ \leftarrow }(x)} I_{ [Z^{ *}_{ (k+1)}>1]}\frac{ dx}{1-U_{ (k+1)} }  \right] \right| \nonumber\\
	&\le k \lambda ^{2 } E \left[   \int_{ U_{ (k+1)}}^{ 1} \left(  \frac{Z^{ *}_{ (k+1)} } { kF_{ Z}^{ \leftarrow }(x)}\right)^{ 2} I_{ [Z^{ *}_{ (k+1)}>1]} \frac{ dx}{1-U_{ (k+1)} } \right] \nonumber   \\
	& \le k^{ -1 }\lambda ^{ 2 }   E\Big[ \frac{ 1}{1-U_{ (k+1)} } \Big]  \to 0
\end{align}
if $ k\to \infty$ satisfying $ n/k\to \infty$ and $ n/k^{ 2}\to 0$. Therefore,  for such a sequence $ k$ we get  
\begin{equation}\label{eq:onlyk:xineg} 
 	E\Big[ \exp \left( -\lambda V_{ n}  \right)  \Big] =   E \left[\exp \left(- \int_{ U_{ (k+1)}}^{ 1}\lambda \frac{Z^{ *}_{ (k+1)} }{ F_{ Z}^{ \leftarrow }(x)} I_{ [Z^{ *}_{ (k+1)}>1]}\frac{ dx}{1-U_{ (k+1)} }\right)  \right]+o(1).
\end{equation}

Since $ V_{ n}\stackrel{ P}{\to }  1-\gamma $ we get 
\begin{equation}\label{eq:stoch:conv:H:xineg} 
 	H(U_{ (k)})\stackrel{ P}{\to  }1-\gamma     
\end{equation}
whenever $ n,k\to \infty$ with $ n/k\to 0$ and $ n/k^{ 2 }\to \infty$, where
\[ 
 	H(y):=\int_{ y}^{ 1} \frac{ F_{ Z}^{ \leftarrow }(y)}{ F_{ Z}^{ \leftarrow }(x) (1-y)}dx .  
\]
The arguments following \eqref{eq:conv:H} gives us 
\[ 
 	b_{ Z}(u):= F_{ Z}^{ \leftarrow }(1-1/u)   \in RV_{ \gamma /(1-\gamma )}
\]
which implies $ \bar F_{ Z}\in RV_{ 1-1/\gamma }$ and that completes the proof.
\end{proof}

\subsection{Gumbel case}\label{subsec:gumbel} For a converse to
\eqref{eqn:sidxi=0}, we found it difficult to
deal with dividing by $X_{(\lceil    k/2      \rceil )}
-X_{(k)}$. However, we were expecting $\Pi$-varying behavior for this
difference and expected this difference to be of the order of a slowly
varying auxiliary function familiar in the theory of $\Pi$-varying
functions. In \eqref{eqn:sidxi=0}, if we replace division by  $X_{(\lceil    k/2      \rceil )}
-X_{(k)}$ with division by  a slowly varying function, the following
partial converse of \eqref{eqn:sidxi=0}
 emerges, which represents a generalization of a
result of \citet{mason:1982}.

\begin{theorem}\label{thm:xizero} 
Suppose $ X_{ 1},\ldots,X_{ n}$ is an iid sample from a distribution $ F$ satisfying $ E[X_{ 1}^{ +}]<\infty$. 
Suppose there exists  $ a(t)\in RV_{ 0}$ such that for every sequence
$ k:=k_{ n} \to \infty$ with $ n/k\to \infty$ 
\begin{equation}\label{eq:MEplotconv:xizero} 
 	\frac{ 1}{ka(n/k) }\sum_{ i=1}^{ k} \Big( X_{ (i)}-X_{ (k+1)} \Big)    \stackrel{ P}{\to } 1.
\end{equation}
Then $ F\in MDA(\Lambda )$, i.e., $ F$ is in the maximal domain of attraction of the Gumbel distribution.
\end{theorem}

\begin{proof} 
We begin by observing that without loss of any generality we can take the function $ a(t)$ to be continuous; see Karamata's repreentation \citep[Corollary 2.1, p.29]{resnick2006htp}. Following the notation used in the proof of Theorem \ref{thm:xipos}, set $ X_{ i}^{ *}=F^{ \leftarrow }(U_{ i})\stackrel{ d}{= }X_{ i} $ where $ U_{ 1},\ldots,U_{ n}$ are iid $ U(0,1)$. For any $ 0<x<1$, define $ Z_{ i}(x):=F^{ \leftarrow }(V_{ i}(x)),i\ge 1$, where $ V_{ i}(x), i\ge 1,$ are iid $ U[x,1]$. Then 
  \begin{align*} 
 	\frac{ 1}{ ka(n/k)}   \sum_{ i=1}^{ k} \Big( X_{ (i)}-X_{ (k+1)} \Big) & \stackrel{ d}{= } \frac{ 1}{ ka(n/k)}   \sum_{ i=1}^{ k} \Big( X^{ *}_{ (i)}-X^{ *}_{ (k+1)} \Big)\\
	& \stackrel{ d}{= }  \frac{ 1}{ ka(n/k)}   \sum_{ i=1}^{ k} \Big(Z_{ i}(U_{ (k+1)})-X^{ *}_{ (k+1)} \Big) \stackrel{ P}{\rightarrow  } 1. 
\end{align*}
Using \eqref{eq:MEplotconv:xizero} we get for any $ \epsilon >0$
\begin{align*} 
 	&P \left[ \Big| \frac{1 }{ka(n/k) }\sum_{ i=1}^{ k} \Big( X_{ (i)}-X_{ (k+1)} \Big) -1 \Big|>\epsilon  \right]   \\
	& E \left[P \left[   \Big| \frac{1 }{ka(n/k) }\sum_{ i=1}^{ k} \Big( Z_{ i}(U_{ (k+1)})-X^{ *}_{ (k+1)} \Big) -1 \Big|>\epsilon \Bigg| U_{ (k+1)} \right] \right] \rightarrow 0 
\end{align*}
which implies
\[ 
 	P \left[   \Big| \frac{1 }{ka(n/k) }\sum_{ i=1}^{ k} \Big( Z_{ i}(U_{ (k+1)})-X^{ *}_{ (k+1)} \Big) -1 \Big|>\epsilon \Bigg| U_{ (k+1)} \right] \stackrel{ P}{ \rightarrow  }  0   .
\]
Then for a subsequence $ k^{ \prime }:=k^{ \prime }_{ n}$ of $ k_{ n}$ we have 
\[ 
  	P \left[   \Big| \frac{1 }{k^{ \prime }a(n/k^{ \prime }) }\sum_{ i=1}^{ k^{ \prime }} \Big( Z_{ i}(U_{ (k^{ \prime }+1)})-X^{ *}_{ (k^{ \prime }+1)} \Big) -1 \Big|>\epsilon \Bigg| U_{ (k^{ \prime }+1)} \right] \stackrel{ a.s.}{  \rightarrow }     0 .	   
\]
Using relative stability and  \citet[Theorem 2, \S 2.7, p. 140]{gnedenko:kolmogorov:1968} we then get
\begin{equation}\label{eq:relstab1}
	k^{ \prime } P \left[ Z_{ 1}(U_{ (k^{ \prime }+1)}) > X^{ *}_{ (k^{ \prime }+1)} +\epsilon k^{ \prime } a(n/k^{ \prime })\Big| U_{ (k^{ \prime }+1 )}\right] \stackrel{ a.s.}{\rightarrow  }  0
\end{equation}
and 
\begin{equation}\label{eq:relstab2}
	k^{ \prime } E \left[ \frac{ Z_{ 1}(U_{ (k^{ \prime }+1)}) - X^{ *}_{ (k^{ \prime }+1)}}{k^{ \prime }a(n/k^{ \prime }) }I_{ [0  \le Z_{ 1}(U_{ (k^{ \prime }+1)}) - X^{ *}_{ (k^{ \prime }+1)} \le \epsilon k^{ \prime } a(n/k^{ \prime })]} \Bigg| U_{ (k^{ \prime }+1)}\right] \stackrel{ a.s.}{\rightarrow  } 1. 
\end{equation}
From \eqref{eq:relstab2} we get 
\[ 
 	\frac{ 1}{a(n/k^{ \prime }) (1-U_{ (k^{ \prime }+1)}) }   \int_{ 0\le F^{ \leftarrow }(s) - X^{ *}_{ (k^{ \prime }+1)} \le \epsilon k^{ \prime } a(n/k^{ \prime }) } F^{ \leftarrow }(s) ds \stackrel{ a.s.}{\to } 1 
\]
and then using \eqref{eq:relstab1} and the assumption that $ E\big[ X_{ 1}^{ +} \big]< \infty $, we obtain
\begin{equation}\label{eq:finalmostsure}
	\frac{ 1}{a(n/k^{ \prime }) (1-U_{ (k^{ \prime }+1)}) } \int_{ U_{ (k^{ \prime }+1)}}^{ 1} F^{ \leftarrow }(s) ds \stackrel{ a.s.}{\to } 1. 
\end{equation}
Observe that this also implies 
\begin{equation}\label{eq:finalinP}
	\frac{ 1}{a(n/k) (1-U_{ (k+1)}) } \int_{ U_{ (k+1)}}^{ 1} F^{ \leftarrow }(s) ds \stackrel{ P}{\to } 1. 
\end{equation}

Now set 
\begin{equation}\label{eq:H:xizero}
	H(x):= \frac{ 1}{ 1-x} \int_{ x}^{ 1} F^{ \leftarrow }(s) ds \ \ \ \ \mbox{ for all  } 0<x<1
\end{equation}
and then \eqref{eq:finalinP} implies
\begin{equation}\label{eq:Hdiscretelimit} 
 	\frac{ H(U_{ (k+1)})}{  a(n/k)} \stackrel{ P}{ \to} 1.    
\end{equation}
We now prove that \eqref{eq:Hdiscretelimit} implies
\begin{equation}\label{eq:Hcontlimit}
	g(t):=\frac{ H(1-1/t)}{a(t) } \to 1 \ \ \ \ \mbox{ as } t\to \infty
\end{equation}
and for that we use the same technique used in the proof of Theorem \ref{thm:xipos}. If \eqref{eq:Hcontlimit} is not true then given any $ \delta >0$ we can get a sequence $1<t^{ (2)}_{ m}\to \infty $ such that 
\[ 
 	   \Big| g\big(t^{ (2)}_{ m}\big)  -1 \Big|> 2\delta .
\]
Using the continuity of $ H(\cdot)$ and $ a(\cdot)$  we can get $ 1<t^{ (1)}_{ m}<t^{ (2)}_{ m}$ with $ t^{ (1)}_{ m}\to \infty$ such that 
\[ 
 	\left| \frac{ H \left( 1-1/y \right) }{ a ( x) } -1 \right|   >\delta \ \ \ \ \mbox{ for all }x, y\in \big(t^{ (1)}_{ m},t^{ (2)}_{ m}\big).
\] 
Now for every $ m\ge 1$ get $ n(m)$ large enough such that 
\[ 
 	  n(m)\ge \big(t^{ (1)}_{ m} \big)^{ 2} \ \ \ \ \mbox{ and } \ \ \ \ n(m) \left( \frac{ 1}{t^{ (2)}_{ m} } -\frac{ 1}{ t^{ (1)}_{ m} } \right) \ge 1+ \sqrt{ \frac{ n(m)}{  t^{ (1)}_{ m} }}.
\]
Also define $ k(n(m))= \lfloor n(m)/t^{ (1)}_{ m} +1\rfloor$ and note that for all $ m\ge 1$ we get  $ n(m)/k(n(m)) \in \big( t^{ (1)}_{ m},t^{ (2)}_{ m} \big) $ and 
\[ 
 	   y^{ (1)}_{ m}:= 1- \frac{ k(n(m))}{n(m) } \ge  1-\frac{ 1}{t^{ (1)}_{ m} }\ \ \ \ \mbox{ and } \ \ \ \ y^{ (2)}_{ m} := 1-\frac{ k(n(m))}{ n(m)} + \frac{ \sqrt{k(n(m))}}{ n(m) } \le 1-\frac{ 1}{t^{ (2)}_{ m} }
\] As in the proof of Theorem \ref{thm:xipos}, write 
\[ 
 	N_{ n}=\frac{ n}{\sqrt{k} } \left( U_{ (k+1)} - \left( 1-\frac{ k}{n } \right) \right)    
\]
and then $ N_{ n}\Rightarrow N(0,\sigma ^{ 2})$ for some $ \sigma ^{ 2}>0$ from \cite{balkema1975limit}. Now observe that with this construction
\begin{align*}
	& \liminf_{ m \rightarrow\infty  } P \left[ \left| a \left( \frac{ n(m)}{k(n(M)) } \right)^{ -1} H \left( \frac{ \sqrt{k(n(m))}}{n(m) }N_{ n(m)}+1-\frac{ k(n(m))}{n(m) } \right) - 1 \right| >\delta  \right]   \\ 
	&  \ge  \liminf_{ m \rightarrow\infty  } P \left[  \frac{ \sqrt{k(n(m))}}{n(m) }N_{ n(m)}+1-\frac{ k(n(m))}{n(m) }\in \left(  1-\frac{ 1}{t^{ (1)}_{ m} }, 1-\frac{ 1}{t^{ (2)}_{ m} }  \right)   \right] \\
	& \ge  \liminf_{ m \rightarrow\infty  } P \left[  \frac{ \sqrt{k(n(m))}}{n(m) }N_{ n(m)}+1-\frac{ k(n(m))}{n(m) }\in \left(y_{ m}^{ (1)},y_{ m}^{ (2)} \right)   \right] \\
	& = \liminf_{ m \rightarrow \infty  } P \left[ N_{ n(m)}\in (0,1) \right] = P[N(0,\sigma ^{ 2})\in (0,1)] >0
 \end{align*}
 which contradicts \eqref{eq:Hdiscretelimit}.
 
 In the last step of the proof we show that \eqref{eq:Hcontlimit} implies $ F\in MDA(\Lambda )$. Let $ \kappa :=F^{ \leftarrow }(1)$ denote the right end point of $ F$. Set 
 \begin{equation}\label{eq:f:defn} 
 	f(x)= \int_{ x}^{ \kappa } \frac{\bar F(s)}{\bar F(x)} ds   \ \ \ \ \mbox{ for all } x<\kappa ,
\end{equation}
and note that \eqref{eq:Hcontlimit} implies $ f(b(t))\sim a(t) \in RV_{ 0}$ where $ b(t):= F^{ \leftarrow }(1-1/t)$. Also observe that \eqref{eq:f:defn} is equivalent to 
\begin{equation}\label{eq:fvonMisesform}
	f(x)\bar F(x)=\int_{ x}^{ \kappa }\bar F(s)ds = c\exp \left\{ -\int_{ 1}^{ x} \frac{ 1}{f(s) }ds \right\} \ \ \ \ \mbox{ for all } x<\kappa 
\end{equation}
for some $ c>0$. We claim that $\int_{ x}^{ \kappa }\bar F(s)ds $ is tail equivalent to a distribution in the Gumbel maximal domain of attraction, i.e., there exists a distribution $ F_{ 1}\in MDA(\Lambda )$ such that
\[ 
 	\frac{ \int_{ x}^{ \kappa }\bar F(s)ds}{\bar F_{ 1}(x) } \to 1 \ \ \ \ \mbox{ as } x\to \kappa .   
\]  
Following \citep[Proposition 0.10, p.28]{resnick:1987} it suffices to check that 
\[ 
 	V(x):= \frac{ c}{ \int_{ x}^{ \kappa }\bar F(s)ds} =  \exp \left\{ \int_{ 1}^{ x} \frac{ 1}{f(s) }ds \right\} \in \Gamma
\]
or $ V^{ \leftarrow }\in \Pi$. By \citep[Proposition 0.11, p.30]{resnick:1987} we know that it is enough to verify $ \left( V^{ \leftarrow } \right)    ^{ \prime }\in RV_{ -1}$. Observe that
\[ 
 	   \left( V^{ \leftarrow }(x) \right)    ^{ \prime }= \frac{ 1}{ V^{ \prime } \left( V^{ \leftarrow} (x) \right)  } =  \frac{ f \left( V^{ \leftarrow }(x) \right) }{ V \left( V^{ \leftarrow }(x) \right)  } \sim  \frac{ f \left( V^{ \leftarrow }(x) \right) }{ x   } \sim \frac{ a \left( b^{ \leftarrow } \left( V^{ \leftarrow }(x) \right)  \right) }{x }.
\]
Since 
\[ 
 	V(b(x))= \frac{ 1}{\bar F(b(x)) f(b(x)) } \sim \frac{ x}{a(x) } \in RV_{ 1}   
\]
we get that $ b^{ \leftarrow } \left( V^{ \leftarrow }(x) \right) \in RV_{ 1}$. Furthermore, since $ a\in RV_{ 0}$ this implies  $ a \left( b^{ \leftarrow } \left( V^{ \leftarrow }(x) \right)  \right) \in RV_{ 0}$ and $a \left( b^{ \leftarrow } \left( V^{ \leftarrow }(x) \right)  \right) /x \in RV_{ -1}$. By \citep[Proposition 0.11, p.30]{resnick:1987} this implies $ V^{ \leftarrow }\in \Pi$ with auxiliary function 
$ a \left( b^{ \leftarrow } \left( V^{ \leftarrow }(x) \right)  \right)$ and \citep[Proposition 0.9, p.27]{resnick:1987} then gives us that $ V\in \Gamma $ with auxiliary function 
\[ 
 	a(b^{ \leftarrow }(V^{ \leftarrow }(V(x))))\sim a(b^{ \leftarrow }(x)) \sim f(b(b^{ \leftarrow }(x)))\sim f(x).   
\]
This implies that $ f(x)$ is a suitable auxiliary function for $\int_{
  x}^{ \kappa }\bar F(s)ds $. From de Haan theory  (\citet[Proposition 1.9,
p.48]{resnick:1987}, \citet{dehaan:1970}) we know that  
\[ 
 	\frac{ \int_{ x}^{ \kappa }\int_{ s}^{ \kappa } \bar F(y)dyds}{ \int_{ x}^{ \kappa }\bar F(s)ds }   
\]
is an auxiliary function for  $\int_{ x}^{ \kappa }\bar F(s)ds
$. Furthermore, \citep[Proposition 1.9, p.48]{resnick:1987}, we also have
\[ 
 	\frac{ \int_{ x}^{ \kappa }\int_{ s}^{ \kappa } \bar F(y)dyds}{ \int_{ x}^{ \kappa }\bar F(s)ds }   \sim f(x) = \frac{ \int_{ x}^{ \kappa }\bar F(s)ds}{ \bar F(x) }
\]
which proves that $ F\in MDA(\Lambda )$ \citep{dehaan:1970}. 
\end{proof}

\end{section}

\section{Acknowledgement}
A long time ago in a galaxy far away, Gennady Samorodnitsky provided
assistance with a subsequence argument that was very helpful for the
present paper.

\bibliographystyle{elsart-harv}
\bibliography{bibfile}
\end{document}